\documentclass[10pt, letterpaper]{article}
\usepackage{fullpage}
\usepackage{amsmath, amssymb,amsthm}

\newtheorem{theorem}{Theorem}[section]
\newtheorem*{thmk1}{Theorem K1}
\newtheorem*{thmk2}{Theorem K2}
\newtheorem{lemma}[theorem]{Lemma}
\newtheorem{proposition}[theorem]{Proposition}
\newtheorem{corollary}[theorem]{Corollary}
\numberwithin{equation}{section}

\newenvironment{definition}[1][Definition]{\begin{trivlist}
\item[\hskip \labelsep {\bfseries #1}]}{\end{trivlist}}

\newenvironment{remark}[1][Remark]{\begin{trivlist}
\item[\hskip \labelsep {\bfseries #1}]}{\end{trivlist}}

\def\P{{\mathbb P}}        
\def\E{{\mathbb E}}        
\def\Z{{\mathbb Z}}         
\def\F{{\cal F}}                 
\def\N{{\mathbb N}}       
\def\R{{\mathbb R}}       
\def\1{{\mathbf 1}}         
\def\S{{\cal S}}               
\def\A{{\cal A}}               
\def\B{{\cal B}}               
\def\Eset{{\cal E}}        
\def\I{{\cal I}}                  
\def\Var{{\mathbf {Var}\,}}    

\title{$L^1$ Ergodic Theorems for Random Group Averages}
\author{Patrick LaVictoire}

\begin{document}
\maketitle
\begin{abstract}
This is an earlier, but more general, version of "An $L^1$ Ergodic Theorem for Sparse Random Subsequences".  We prove an $L^1$ ergodic theorem for averages defined by independent random selector variables, in a setting of general measure-preserving group actions.  A far more readable version of this paper is in the works.
\end{abstract}
\section{Introduction}
Let $(X,\F,m)$ be a non-atomic probability space and $T$ a measure-preserving transformation on $X$; we call $(X,\F,m,T)$ a (discrete) dynamical system.  For a sequence of integers ${\mathfrak n}=\{n_k\}$ and any $f\in L^1(X)$, we may define the subsequence average
$$A_N^{({\mathfrak n})}f(x):= \frac1N \sum_{k=1}^N f(T^{n_k}x).$$
Given a sequence ${\mathfrak n}$, a major question is for which $1\leq p\leq\infty$ and which $(X,\F,m,T)$ we have convergence of various sorts for all $f\in L^p(X)$.  Two important definitions along these lines are as follows:
\begin{definition}
A sequence of integers ${\mathfrak n}=\{n_k\}$ is \emph{universally $L^p$-good} if for every dynamical system $(X,\F,m, T)$ and every $f\in L^p(X,m)$, $\displaystyle\lim_{N\to\infty} A_N ^{({\mathfrak n})} f(x)$ exists for almost every $x\in X$.\\ 
A sequence of integers ${\mathfrak n}=\{n_k\}$ is \emph{universally $L^p$-bad} if for every nontrivial ergodic dynamical system $(X,\F,m, T),$ there exists an $f\in L^p(X,m)$ such that $\{A_N ^{({\mathfrak n})} f(x)\}_{N=1}^\infty$ diverges on a set of positive measure.
\end{definition}
\noindent Birkhoff's Ergodic Theorem asserts, for instance, that the sequence $n_k=k$ is universally $L^1$-good.  On the other extreme, Rosenblatt \cite{JR2} proved that the sequence $n_k=2^k$ (or any lacunary sequence) is universally $L^\infty$-bad (and even worse, see for example \cite{SSO}).  Between these extrema lie many results on the existence of universally $L^p$-good sequences of various sorts, beginning with Bourgain's celebrated result \cite{JB2} that if $a(x)$ is an integer-valued polynomial, then $\{a(k)\}$ is universally $L^2$-good; in \cite{JB3}, he extended this to integer parts of real-valued polynomials and all $p>1$.  Boshernitzan \emph{et al.} \cite{EAS} proved several results characterizing smooth subpolynomial functions $a(x)$ whose integer parts $\lfloor a(n)\rfloor$ are universally $L^2$-good.
\\
\\ The most restrictive case $p=1$ is of particular interest because the positive results above do not extend to this case.  A surprising illustration of the difference is the recent result of Buczolich and Mauldin that $n_k=k^2$ is universally $L^1$-bad \cite{DSA}.  Positive results in $L^1$ have been difficult to come by, particularly for sequences which are sparse in $\N$.
\\
\\ Universally $L^1$-good sequences of density 0 had long been known to exist, but these were sparse block sequences, which consist of large 'blocks' of consecutive integers, separated by wide gaps. Bellow and Losert \cite{BL} showed that for any $F: \N\to\R^+$, there exists a universally $L^1$-good block sequence $\{n_k\}$ with $n_k\geq F(k)$.  To distinguish such block sequences from more uniformly distributed ones, we recall the notion of Banach density:
\begin{definition}
A sequence of positive integers $\{n_k\}$ has Banach density $c$ if $$\lim_{m\to\infty} \sup_{N} \frac{|\{n_k\in[N,N+m)\}|}{m}=c.$$
\end{definition}
Note that block sequences with arbitrarily large block lengths have Banach density 1 (the sequences in \cite{BL} are all of this sort).  It was at first conjectured in \cite{PETHA} that there existed no universally $L^1$-good sequences with Banach density 0.  However, Buczolich \cite{GTI} has constructed a (slowly growing) counterexample, and Urban and Zienkiewicz \cite{UZ} subsequently proved that the sequence $\lfloor k^a \rfloor$ for $1<a<1+\frac1{1000}$ is universally $L^1$-good.
\\
\\ Bourgain \cite{JB2} noted that certain sparse random sequences were universally $L^p$-good with probability 1 for all $p>1$.  These sequences are generated as follows: given a decreasing sequence of probabilities $\{\tau_j:j\in\N\}$, let $\{\xi_j:j\in\N\}$ be independent random variables on a probability space $\Omega$ with $\P(\xi_j=1)=\tau_j,\;\P(\xi_j=0)=1-\tau_j$.  Then for each $\omega\in\Omega$, define a random sequence by taking the set $\{n:\xi_n(\omega)=1\}$ in increasing order.  (For $\alpha>0$ and $\tau_j=O(j^{-\alpha})$, these sequences have Banach density 0 with probability 1; see Prop. \ref{banach} of this paper.)
\\
\\In their treatment \cite{PETHA} of Bourgain's method, Rosenblatt and Wierdl demonstrate by Fourier analysis that if $\tau_j\to0$ slowly enough (e.g. $\tau_j\geq \frac{c(\log \log j)^{1+\epsilon}}{j}$ suffices), then $\{n:\xi_n(\omega)=1\}$ is universally $L^2$-good with probability 1 (see Example 4.7), thus proving the existence of superpolynomial universally $L^2$-good sequences.  However, their approach cannot be applied to the $L^1$ case.
\\
\\ In this paper, we use the probabilistic method in conjunction with a construction of \cite{UZ} to achieve the following $L^1$ result:

\begin{theorem}
\label{main}
Let $0<\alpha<1/2$, and let $\xi_n$ be independent selector variables on $\Omega$ with $\P(\xi_n=1)=n^{-\alpha}$.  Then there exists a set $\Omega'\subset\Omega$ of probability 1 such that for every $\omega\in\Omega'$, $\{n:\xi_n(\omega)=1\}$ is universally $L^1$-good.
\end{theorem}

\noindent Thus we prove the existence of universally $L^1$-good sequences which grow much more rapidly than the ones obtained in \cite{UZ} or \cite{GTI}, and which grow uniformly as compared to the sparse block sequences of \cite{BL}.
\\
\\As our method does not make use of the Fourier transform, it also extends to the more general case of measure-preserving group actions, if the group has a polynomial rate of growth in the sense of Bass \cite{HB}.  This generality does not make the proofs more difficult, and even the $L^2$ theorem is new in this context, so we will work in this more general case.  Theorem \ref{main} will be a particular case of Theorem \ref{L1}, the statement of which requires the notation of Chapter 2.
\\
\\ In Section 5 we generalize this approach in order to take random subsequences of universally $L^1$-good sequences.  In particular, we use the sparse block sequences of \cite{BL} to achieve the following:
\\
\\ \textbf{Theorem \ref{faster}.} \emph{For every $F:\N\to\R^+$, there exists a universally $L^1$-good sequence $\{n_k\}$ with $n_k\geq F(k)$ and Banach density 0.}

\section{Averages for Measure-Preserving Group Actions}

We begin with a few necessary definitions for the more general case.  Those less intrigued by the case of general semigroups of operators may prefer to follow the Remarks below, which describe the notation for $G=\Z^d$; that is, for averages of powers of $d$ commuting measure-preserving transformations.
\begin{definition}
Let $G$ be an infinite finitely generated group with identity $e$, and ${\cal A}=\{a_1,\dots,a_n\}\subset G$.  Let $\S^{\cal A}_N:=\bigcup^N_{k=0} {\cal A}^k$ denote the elements of $G$ expressible as words of length $N$ or less in $\cal A$, counting $\A^0:=\{e\}$.  Let $\S^{\cal A}\subset G$ be the semigroup $\bigcup^\infty_{k=0} {\cal A}^k$ generated by $\cal A$.  For $g\in\S^{\cal A}$, let $\rho^{\cal A}(g):=\min\{N:g\in \S^{\cal A}_N\}$; if $\S^{\cal A}=G$, then $\rho(g,h)=\rho^{\cal A}(gh^{-1})$ is a ($\N$-valued) metric on $G$.
\end{definition}
\begin{remark}[Remark 1]
For $G=\Z^d$, we will use the ordinary basis set $\A=\{e_1,\dots,e_d\}$; then $\S_N=\{0,\dots,N\}^d$ and $\rho^{\A}(n_1,\dots,n_d)=\sup n_i$.
\end{remark}
\begin{definition}
Let $G$ be a group generated by the finite set $\cal B$.  We say that $G$ has \emph{polynomial growth} if there is some $d\in\N$ such that $|\S^{\B}_N|= O(N^d)$, and \emph{polynomial growth of degree $d$} if $|\S^{\B}_N|={\mathbf\Theta}(N^d)$ (i.e. for some $C>0$, $C^{-1}N^d\leq|\S^{\B}_N|\leq CN^d$ for all $N\gg0$).
\end{definition}
In fact, these definitions are independent of $\cal B$, and any group with polynomial growth has polynomial growth of degree $d$ for some $d\in\N$ (see VI.2 of \cite{AGG}).  Since any finite ${\cal A}\subset G$ can be extended to a generating set, we then have the upper bound $|\S^{\A}_N|\leq CN^d$.  We will need more than this at first; we will assume that 
\begin{eqnarray}
\label{jack}
&|\S^{\A}_N|={\mathbf\Theta}(N^d),\\
\label{jill}
&|\S^\A_N\Delta g\S^\A_N|=o_g(|\S^{\A}_N|)\;\forall g\in\S.
\end{eqnarray}
\noindent These assumptions may in fact be redundant.  In the case where $\A=\A^{-1}$ generates the group $G$ with polynomial growth of degree $d$ (so $\S^\A=G$), Pansu \cite{Pan} proves that $|\S^{\A}_N|N^{-d}\to \tilde C\in(0,\infty)$, which implies (\ref{jack}) and (\ref{jill}).  However, we have searched the literature in vain for similar results on growth functions of semigroups rather than groups.
\\
\\ From this point, we will take the set $\cal A$ as given, suppressing it in superscripts.  Without loss of generality, we may assume that $G$ is the group generated by $\A$ and $\S$ is the semigroup generated by $\A$.
\\ 
\\ Let $\{\tau_n:n\in\N\}$ be a nonincreasing sequence of probabilities.  Let $\Omega$ be a probability space, and $\{\xi_g(\omega): g\in\S\}$ be independent mean $\tau_{\rho(g)}$ Bernoulli random variables on  $\Omega$: $\P(\xi_g=1)=\tau_{\rho(g)}$ and $\P(\xi_g=0)=1-\tau_{\rho(g)}$.  Let
$$\beta(N):=\sum_{g\in\S_N}\tau_{\rho(g)}.$$
If $\beta(N)\to\infty$, then $\P\left(\beta(N)^{-1}\sum_{g\in\S_N}\xi_g\to1\right)=1$ (see the Remark in Section 5).  We restrict ourselves to this set $\Omega_1$ of probability 1, and  remark that for a power law $\tau_n=n^{-\alpha}$ on a group with polynomial growth of degree $d$, we have $\beta(N)={\mathbf\Theta}(N^{d-\alpha})$ for $\alpha<d$; if $\alpha>d$, then $\P(|\{g\in \S:\xi_g=1\}|<\infty)=1$.
\begin{definition}
Let $(X,\F,m)$ be a probability space and $\{T_g:g\in G\}$ a group of measure-preserving transformations on $X$ with $T_gT_h=T_{gh}$ for all $g,h\in G$.  We say that $\{T_g\}$ is a \emph{measure-preserving group action}.
\end{definition}
\begin{definition} For a measure-preserving group action $(X,\F,m, \{T_g\})$ and$f\in L^1(X)$, define the average $$A_Nf(x):=|\S_N|^{-1}\sum_{g\in\S_N}f(T_gx)$$ and the random average $$A_N^{(\omega)}f(x):=\beta(N)^{-1}\sum_{g\in\S_N}\xi_g(\omega)f(T_gx).$$
\end{definition}
\begin{remark}[Remark 2]
In the case $G=\Z$, ${\cal A}=\{1\}$, we have $\beta(N)=\sum_{j=1}^N\tau_j$; for $\omega\in\Omega_1$, the a.e. convergence of $A_N^{(\omega_0)}f=\beta(N)^{-1}\sum_{j=1}^N\xi_j(\omega_0)f(T^jx)$ for every dynamical system $(X,\F,m,T)$ and every $f\in L^p(X)$ is equivalent to the statement that $\{j\in\N: \xi_j(\omega_0)=1\}$ is universally $L^p$-good.
\end{remark}
Krengel proves several theorems about measure-preserving group actions and other additive processes in Section 6.4 of \cite{UK}.  We will apply Theorems 4.1, 4.2 and 4.4 in that section to our particular case:
\begin{thmk1}
\label{K}
Let $G$ have polynomial growth of degree $d$, ${\cal A}\subset G$ finite satisfying (\ref{jack}) and (\ref{jill}).  Then for every measure-preserving group action $(X,\F,m,\{T_g\})$ and $1\leq p<\infty$, $A_Nf$ converges in $L^p$ and a.e. for every $f\in L^p(X,m)$.
\end{thmk1}
\begin{thmk2}
\label{K2}
Let $G$ have polynomial growth of degree $d$, ${\cal A}\subset G$ finite satisfying (\ref{jack}) and (\ref{jill}).  Then we have a weak-type maximal inequality on $G$ itself,
\begin{eqnarray}
\label{costello}
|\{g\in G: \sup_N |\varphi\ast|\S_N|^{-1}\1_{\S_N}|>\lambda\}|\leq\frac{C}\lambda\|\varphi\|_1\hspace{10pt}\forall \varphi\in\ell^1(G).
\end{eqnarray}
\end{thmk2}
\noindent We may now state our main results:
\begin{theorem}
\label{L2}
Let $G$ be a finitely generated group with polynomial growth of degree $d$, ${\cal A}\subset G$ finite satisfying (\ref{jack}) and (\ref{jill}), $0<\alpha<d$, and $\tau_n=n^{-\alpha}\,\forall n>0$.  Then there exists $\Omega_2\subset\Omega$ with $\P(\Omega_2)=1$ such that for each $\omega\in\Omega_2$, $A_N^{(\omega)}f$ converges in $L^2$ and a.e. for every measure-preserving group action $(X,\F,m,\{T_g\})$ and every $f\in L^2(X,m)$.
\end{theorem}
\begin{theorem}
\label{L1}
Let $G$ be a finitely generated group with polynomial growth of degree $d$, ${\cal A}\subset G$ finite satisfying (\ref{jack}) and (\ref{jill}), $0<\alpha<d/2$, and $\tau_n=n^{-\alpha}\,\forall n>0$.  Then there exists $\Omega_3\subset\Omega$ with $\P(\Omega_3)=1$ such that for each $\omega\in\Omega_3$, $A_N^{(\omega)}f$ converges in $L^1$ and a.e. for every measure-preserving group action $(X,\F,m,\{T_g\})$ and every $f\in L^1(X,m)$.
\end{theorem}
\begin{remark}[Remark 3]
Theorem \ref{main} is then a direct application of Theorem \ref{L1} in the case $G=\Z$, $\S=\N$; given any dynamical system, there exists a $\Z$-action for which the averages $A_N^{(\omega)}f$ have the same distribution (see e.g. \cite{UK}, Section 1.4).
\end{remark}

\section{Proof of Theorem \ref{L2}}
\noindent As Bourgain \cite{JB2} has already proved this theorem in the case of the integers and the standard ergodic averages, readers uninterested in extending this theorem to more general group actions ought to skip to Section 4.
\\
\\ To prove Theorem \ref{L2}, it will suffice to prove convergence of the averages along a suitable subsequence.  Indeed, fix an increasing sequence $\{a_j\}\subset\N$ such that $\frac{a_{j+1}}{a_j}\to1$.  Let $M_j:=\min\{n:\beta(n)\geq a_j\}$ and $m_j:=M_j-1$; then $\displaystyle\frac{\beta(m_{j+1})}{\beta(M_j)}\leq\frac{a_{j+1}}{a_j}$, and for any $f\geq0$ and $M_j\leq N\leq m_{j+1}$,
$$\frac{a_j}{a_{j+1}}A_{M_j}^{(\omega)}f \leq A_N^{(\omega)}f \leq \frac{a_{j+1}}{a_j}A_{m_{j+1}}^{(\omega)}f.$$
Set $N_{2j-1}=m_j, N_{2j}=M_j$.  Then under the assumptions of Theorem \ref{L2}, it suffices to prove that $A_{N_j}^{(\omega)}f$ converges in $L^2$ and a.e. for all $f\in L^2(X)$.
\\
\\We may assume that the original sequence $\{a_j\}$ is superpolynomial; i.e. $a_j\neq O(j^A)$ for every $A\in\N$; then by assumption (\ref{jack}), we see that $N_j$ is superpolynomial as well.
\\
\\ We will compare these random averages to their expectation, a deterministic weighted average; define 
$$\sigma_Nf(x):=\E_{\omega}A_N^{(\omega)}f(x)=\beta(N)^{-1}\sum_{g\in\S_N}\tau_{\rho(g)}f(T_gx)=\sum_{n=0}^N a_{n,N}A_n f(x),$$
where $a_{n,N}\geq0$, $\displaystyle\sum_{n=0}^Na_{n,N}=1$ for all $N$, and $\displaystyle \lim_{N\to\infty}a_{n,N}=0$ for all $n$.  Since $A_nf$ converges in $L^2$ and a.e. by Theorem K1, clearly $\sigma_Nf$ converges in $L^2$ and a.e. as well.
\\
\\ We will prove Theorem \ref{L2} by showing that there exists a set $\Omega_2\subset\Omega_1$ with $\P(\Omega_2)=1$ such that for every $\omega\in\Omega_2$,
\begin{eqnarray}
\label{hammer}
\|\sup_{j\geq k} |A_{N_j}^{(\omega)}f-\sigma_{N_j}f|\|_2\to0 \text{ as }k\to\infty\;\forall f\in L^2(X),
\end{eqnarray}
which immediately implies $A_{N_j}^{(\omega)}f-\sigma_{N_j}f\to0$ in $L^2$ and a.e.
\\
\\ As in \cite{APET} and other papers, we hope to transfer the corresponding maximal inequality from the group algebra $\ell^p(G)$.  The transference argument is practically identical to the case $G=\Z$, but it is necessary to prove it in this general setting.
\begin{lemma}
\label{transfer}
Let $G$ be a group with polynomial growth, and $(X,\F,m,\{T_g\})$ be a measure-preserving group action; let $\{a_{g,j}\}\subset\mathbb{C}$ such that $\sum_{g\in G} |a_{g,j}|<\infty\;\forall j$. Set $A_jf=\sum_{g\in G} a_{g,j}T_gf$ and $\mu_j=\sum_{g\in G} a_{g,j}\delta_g$.
\\
\\ For any $1\leq p\leq\infty$, if $\|\sup_j |\psi\ast\mu_j|\|_p\leq C_0\|\psi\|_p\;\forall\psi\in\ell^p(G)$, then $\|\sup_j |A_jf|\|_p\leq C_0\|f\|_p\;\forall f\in L^p(X)$;
\\
\\ if instead $\|\sup_j |\psi\ast\mu_j|\|_{p,\infty}\leq C_0\|\psi\|_p\;\forall\psi\in\ell^p(G)$, then $\|\sup_j |A_jf|\|_{p,\infty}\leq C_0\|f\|_p\;\forall f\in L^p(X)$.
\end{lemma}
\begin{proof}
We first consider the strong maximal inequality.  It is enough to show that $\|\sup_{1\leq j\leq J} |A_jf|\|_p\leq C_0\|f\|_p$ for all $f\in L^p(X)$, for each fixed $J\in\N$.  We may further assume that the supports of the $\mu_j$ are finite, and let $\Eset:=\bigcup_{j=1}^J \text{ supp }\mu_j$.  Take a set $\A=\A^{-1}$ that generates $G$, and the sets $\S^\A_N$ defined in Section 2.  Fix $x\in X$ and a large finite $K\in\N$, and define $\varphi$ on $G$ by $\varphi(g)=\left\{ \begin{array}{ll} f(T_{g^{-1}}x) & \mbox{if} \; g^{-1}\in\S_K+\Eset,\\
0 & \mbox{otherwise.}  \end{array}\right.$
\\ Then $A_jf(T_gx)=\varphi\ast\mu_j(g^{-1})$ for all $g\in \S_K$ and all $j\leq J$.  This completes the proof for $p=\infty$; for $p<\infty,$
\begin{eqnarray*}
\sum_{g\in\S_K} \sup_{1\leq j\leq J} |A_jf(T_gx)|^p &=& \sum_{g\in\S_K} \sup_{1\leq j\leq J}|\varphi\ast\mu_j(g^{-1})|^p
\leq\|\sup_{k\leq j\leq J} |\varphi\ast\mu_j|\|_p^p\\
&\leq& C_0^p\|\varphi\|_p^p\\
&=&C_0^p\sum_{g\in\S_K+\Eset}|f(T_gx)|^p.
\end{eqnarray*}
Integrating over $x\in X$,
\begin{eqnarray*}
\|\sup_{1\leq j\leq J}|A_jf|\|_p^p\leq C_0^p\frac{|\S_K+\Eset|}{|\S_K|}\|f\|_p^p;
\end{eqnarray*}
letting $K\to\infty$ and noting that (\ref{jill}) holds in this case (see \cite{Pan}), we obtain
\begin{eqnarray*}
\|\sup_{1\leq j\leq J}|A_jf|\|_p\leq C_0\|f\|_p.
\end{eqnarray*}
For the weak inequality, we similarly derive
\begin{eqnarray*}
\lambda^p|\{g\in\S_K: \sup_{1\leq j\leq J}|A_jf(T_gx)|>\lambda\}|\leq C_0^p\|\varphi\|_p^p
\end{eqnarray*}
and integrate this in the same manner.
\end{proof}
\noindent \textbf{Proof of Theorem \ref{L2} (Continued):} We will transfer this problem to $\ell^2(G)$ using Lemma \ref{transfer}.  Let $\eta_g(\omega)=\xi_g(\omega)-\tau_{\rho(g)}$; these are independent mean 0 Bernoulli variables.  Define for each $j$ the random measures
\begin{eqnarray}
\nu_j^{(\omega)}(g)&=&\left\{\begin{array}{ll} \beta(N_j)^{-1}\eta_{g}(\omega), & g\in\S_{N_j} \\ 0, & g\not\in\S_{N_j}\end{array}\right.
\end{eqnarray}
Then for $\varphi\in \ell^p(G)$, we have the random averages $\varphi\ast\nu_j^{(\omega)}(h)=\beta(N_j)^{-1}\sum_{g\in\S_{N_j}}\xi_g(\omega)\varphi(hg^{-1})$, which correspond to the operators $A_{N_j}^{(\omega)}-\sigma_{N_j}$ in the sense above. Theorem \ref{L2} therefore reduces to verifying that with probability 1 in $\Omega$, there is a sequence $C_{k,\omega}\to0$ such that
\begin{eqnarray}
\label{nail}
\|\sup_{j\geq k} |\psi\ast\nu_j^{(\omega)}|\|_2 \leq C_{k,\omega}\|\psi\|_2\; \forall \psi\in\ell^2(G).
\end{eqnarray}
Since $\|\sup_{j\geq k} |\psi\ast\nu_j^{(\omega)}|\|_2^2\leq\|\sum_{j\geq k} |\psi\ast\nu_j^{(\omega)}|\|_2^2=\sum_{j\geq k} \|\psi\ast\nu_j^{(\omega)}\|_2^2$, it clearly suffices to prove that
\begin{eqnarray*}
\sum_{j=1}^\infty \|\nu_j^{(\omega)}\|^2_{op}\leq\infty,
\end{eqnarray*}
where $\|\cdot\|_{op}$ is the norm of the convolution operator on $\ell^2(G)$.
\\
\\ For any operator $A$ on the Hilbert space $\ell^2(G)$, the operator norm $\|A\|=\|A^*A\|^{1/2}=\|(A^*A)^M\|^{1/2M}$; for the convolution operator $Af=\mu\ast f$, the adjoint operator is simply $A^*f=\tilde\mu\ast f$ for $\tilde\mu(g):=\overline{\mu(g^{-1})}$ ($G$ is discrete, thus unimodular).  Thus we have the trivial bound $\|A\|_{op}\leq\|(\tilde\mu\ast\mu)^M\|_{op}^{1/2M}\leq\|(\tilde\mu\ast\mu)^M\|_{\ell^1}^{1/2M}$.  (Here and in what follows, we use $\mu^n$ to denote the $n$-fold convolution product $\mu\ast\mu\ast\dots\ast\mu$.)
\begin{lemma}
\label{patton}
Let $G$ be a group and $E$ a finite subset.  Let $\{X_g\}_{g\in E}$ be independent random variables with $|X_g|\leq1$ and $\E X_g=0$. Assume that $\sum_{g\in E} \Var X_g \geq1$.  Let $X$ be the random $\ell^1(G)$ function $\sum_{g\in E} X_g\delta_g$.  Then $\E \|(\tilde X\ast X)^M\|_{\ell^2}^2\leq C_M (\sum_{g\in E} \Var X_g)^{2M}$, where $C_M$ depends only on $M$.
\end{lemma}
\begin{proof}
\begin{eqnarray*}
\E(\|(\tilde X\ast X)^M\|_{\ell^2}^2) & = & \E \sum_{g\in G}\left(\sum_{\scriptsize \begin{array}{c} g_1^{\,} h_1^{-1}\dots g_M^{\,}h_M^{-1}=g \\ g_i^{\,},h_i^{\,}\in E \end{array}}  X_{g_1^{\,}} X_{h_1^{\,}}\dots X_{g_M^{\,}} X_{h_M^{\,}} \right)^2\\
& = & \sum_{\scriptsize \begin{array}{c} g_1^{\null}h_1^{-1}\dots g_M^{\,}h_M^{-1}=g_{M+1}^{\null} h_{M+1}^{-1}\dots g_{2M}^{\,}h_{2M}^{-1}\\ g_i^{\,},h_i^{\,}\in E\end{array}}\E( X_{g_1^{\,}} X_{h_1^{\,}}\dots X_{g_{2M}^{\,}} X_{h_{2M}^{\,}})
\end{eqnarray*}
For any of these terms, if some $g\in  E$ appears exactly once among the $g_i$ and $h_j$, the expectation of the term will equal 0 by the independence of the $ X_g$.  Therefore we can sort the remaining terms based on the equalities between various $g_i$ and $h_{j}$; namely, in correspondence with the set partitions of $\{1,\dots,4M\}$ in which each component has size $\geq 2$.  Let there be $C_M$ of these.  For a fixed partition $\Lambda=(\lambda_1,\dots,\lambda_q)$, we can majorize the sum 
\begin{eqnarray*}
\sum_{\scriptsize \begin{array}{c} (g_1^{\,},\dots,g_{2M}^{\,},h_1^{\,},\dots,h_{2M}^{\,})\in\Lambda \\ g_i^{\,},h_i^{\,}\in E \end{array}}\E(X_{g_1^{\,}}\dots X_{g_{2M}^{\,}} X_{h_1^{\,}}\dots X_{h_{2M}^{\,}})&\leq& \sum_{g_1^{\,},\dots,g_q^{\,}\in E\text{ distinct}}\E(| X_{g_1^{\,}}|^{|\lambda_1|})\dots\E(| X_{g_q^{\,}}|^{|\lambda_q|})\\
&\leq& \sum_{g_1^{\,},\dots,g_q^{\,}\in E}\E X_{g_1^{\null}}^2\dots\E X^2_{g_q^{\null}}\\
&=&(\sum_{g\in E} \Var X_g)^q\leq(\sum_{g\in E}\Var X_g)^{2M}
\end{eqnarray*}
since $\E |X_g|^p\leq\|X_g\|_\infty^{p-2}\E X_g^2\leq \E X_g^2$ for $p>2$, $\sum_{g\in E} \Var X_g\geq1$ and $q\leq 2M$.
\\ \\ 
Thus $\E(\|(\tilde X\ast X)^M\|_{\ell^2}^2)\leq C_M(\sum_{g\in E} \Var X_g)^{2M}.$
\end{proof}
\noindent \textbf{Proof of Theorem \ref{L2} (Conclusion):} Now by H\"older's Inequality, 
\begin{eqnarray*}
\|(\tilde\nu_j\ast\nu_j)^M\|_1\leq \|(\tilde\nu_j\ast\nu_j)^M\|_2|\text{supp }(\tilde\nu_j\ast\nu_j)^M|^{1/2}\leq \|(\tilde\nu_j\ast\nu_j)^M\|_2|(\S_{N_j}^{-1}\S_{N_j})^M|^{1/2}\leq \|(\tilde\nu_j\ast\nu_j)^M\|_2 C(2MN_j)^{d/2}
\end{eqnarray*}
where $(\S_N^{-1}\S_N)^M:=\{g_1^{-1}h_1\dots g_N^{-1}h_N: g_i,h_i\in \S_N \forall i\}$, using for the last inequality the fact that this is contained in the ball of radius $2MN_j$ about the origin (in the metric $\rho^{\A\cup\A^{-1}}$).  By Lemma \ref{patton}, since $\Var \eta_g= \tau_g(1-\tau_g)\leq\tau_g$,
$$\E(\|(\tilde\nu_j^{(\omega)}\ast\nu_j^{(\omega)})^M\|_{\ell^2}^2)\leq \beta(N_j)^{-4M}\cdot C_M (\sum_{g\in \S_{N_j}} \Var \eta_g)^{2M}\leq C_M \beta(N_j)^{-2M}\leq C_{d,\alpha, M}N_j^{2M(\alpha-d)}$$
and therefore by Chebyshev's Inequality,
\begin{eqnarray*}
\P(\|(\tilde\nu_j^{(\omega)}\ast\nu_j^{(\omega)})^M\|_1>\lambda)& \leq&\P\left( \|(\tilde\nu_j^{(\omega)}\ast\nu_j^{(\omega)})^M\|_2^2 C^2(2MN_j)^{d}>\lambda^2\right)\\ &\leq& C\lambda^{-2}M^dN_j^d\cdot\E(\|(\tilde\nu_j^{(\omega)}\ast\nu_j^{(\omega)})^M\|_{\ell^2}^2)\\
&\leq & C_{d,\alpha,M} \,\lambda^{-2}N_j^{2M\alpha-d(2M-1)}.
\end{eqnarray*}
As $\alpha<d$, take $M,\delta>0$ such that $d(2M-1)>2M\alpha +\delta$.  Take $\lambda=j^{-M(1+\epsilon)}$; since $N_j^\delta$ is superpolynomial, $\sum_j j^{2M(1+\epsilon)}N_j^{-\delta}<\infty$ so by the Borel-Cantelli Lemma, there is a set $\Omega_2\subset\Omega_1$ of probability 1 on which $\|(\tilde\nu_j^{(\omega)}\ast\nu_j^{(\omega)})^M\|_1<C_\omega j^{-M(1+\epsilon)}\,\forall j$ and thus $\sum_{j=1}^\infty \|\nu_j^{(\omega)}\|_{op}^2\leq C_\omega \sum_{j=1}^\infty j^{-1-\epsilon}<\infty$.  This completes the proof of Theorem \ref{L2}.
\begin{remark}
This approach is not limited to polynomial rates of decay; in fact, so long as $\tau_n$ is nonincreasing (or, more generally, as long as we have the convergence and maximal inequality for the averages $\sigma_N$) and $\beta(N)\geq cN^\epsilon$ for some $c,\epsilon>0$, the argument carries through.
\\
\\ In fact, this limit can be pushed slightly by letting the power $M$ depend as well on the index $j$, noting that the constant $C_M$ in Lemma \ref{patton} is bounded by the Bell number $B_{2M}\leq(2M)^{2M}$.  For any $\epsilon>0$, we can show $\tau_n=n^{-d}(\log n)^{1+\epsilon}$ works by taking $0<r<s<(1+\epsilon)r<1$ and setting $N_j=\exp(j^r)$, $M_j=j^s$ and $\lambda_j=(\log j)^{-j^s}$.  On $\Z$, this is still not as strong as the $L^2$ result in \cite{PETHA}; some of this discrepancy can be ascribed to loss in the inequality $\|\hat\nu\|_\infty\leq\|\nu\|_1$.
\end{remark}

\section{Proof of Theorem \ref{L1}}
By Theorem \ref{L2}, for $\omega\in\Omega_2$ we have a.e. convergence of $A_N^{(\omega)}f$ for $f\in L^2(X)$, which is dense in $L^1(X)$.  We therefore need only a weak type maximal inequality to prove Theorem \ref{L1}.  As usual, it is enough to consider the dyadic subsequence $2^j$.  By assumption (\ref{jack}), $\frac{\beta(2^{j+1})}{\beta(2^j)}$ is uniformly bounded and $0\leq A^{(\omega)}_{N}\leq \frac{\beta(2^{j+1})}{\beta(2^j)}A^{(\omega)}_{2^{j+1}}$ for $2^j\leq N<2^{j+1}$, so it suffices to prove
\begin{eqnarray}
\|\sup_j |A^{(\omega)}_{2^j}f|\|_{1,\infty}\leq C\|f\|_1 \;\forall f\in L^1(X).
\end{eqnarray}
Again, we will use Lemma \ref{transfer} to transfer this maximal inequality from $\ell^1(G)$.  Let
\begin{eqnarray*}
\mu_j^{(\omega)}(g)&:=&\left\{\begin{array}{ll} \beta(2^j)^{-1}\xi_{g}(\omega), & g\in\S_{2^j} \\ 0, & g\not\in\S_{2^j}\end{array}\right.\\
\E\mu_j(g)&:=&\left\{\begin{array}{ll} \beta(2^j)^{-1}\tau_{g}, & g\in\S_{2^j} \\ 0, & g\not\in\S_{2^j}\end{array}\right.\\
\nu_j^{(\omega)}(g)&:=&\mu_j^{(\omega)}(g)-\E\mu_j^{(\omega)}(g);
\end{eqnarray*}
$\mu_j^{(\omega)}$ and $\E\mu_j$ correspond to the operators $A_{2^j}^{(\omega)}$ and $\sigma_{2^j}$, respectively.  Theorem \ref{L1} reduces to proving
\begin{eqnarray}
\| \sup_j | \varphi\ast\mu_{j}^{(\omega)} |\|_{1,\infty}\leq C_\omega\| \varphi\|_1.
\end{eqnarray}
\noindent The heart of this proof is the generalization of a deterministic argument from the paper by Urban and Zienkiewicz \cite{UZ}, related to a construction of Christ in \cite{MC2}:
\begin{proposition}
\label{shine on}
Let $\mu_j$ and $\nu_j$ be sequences of functions in $\ell^1(G)$, where $G$ has polynomial growth of degree $d$.  Let $r_j:=|\emph{supp }\mu_j|$ and take $R_j:=\inf\{R>0: \nu_j(g)\neq0\implies \rho(g,e)\leq R\}$.  Assume there exists $C_0<\infty$ such that $\sum_{j\leq k} r_j\leq C_0r_k\;\forall k\in\N$, and that
\begin{eqnarray}
\label{convo}
\nu_j\ast\tilde\nu_j=O(r_j^{-1})\delta_e+O(R_j^{-d-\epsilon})\text{ for some }\epsilon>0.
\end{eqnarray}
If $\forall \varphi$, $\| \displaystyle\sup_j \varphi\ast |\mu_j-\nu_j|\|_{1,\infty}\leq C\|\varphi\|_1$ and $\| \displaystyle\sup_j|\varphi\ast\mu_j|\|_{p,\infty}\leq C_p\|\varphi\|_p$ for some $1<p\leq\infty,$ then 
\begin{eqnarray}
\label{sup}
\|\sup_j|\varphi\ast\mu_j|\|_{1,\infty}\leq C'\|\varphi\|_1 \;\forall\varphi\in\ell^1(G).
\end{eqnarray}
\end{proposition}
\begin{proof}
We will follow the argument in Section 3 of \cite{UZ}, which makes use of a Calderon-Zygmund type decomposition of $\varphi$ depending on the index $j$; however, we must show that this makes sense on more general groups $G$.
\\
\\ We make use of the $\rho$-dyadic cubes constructed by Christ in \cite{MC}.  Namely, there exist a collection of subsets $\{Q_{s,k}\subset G: s\in \N, k\in\Z\}$, and constants $A>1, a_0>0, C_1<\infty$ such that
\begin{eqnarray}
&\text{For each }s\in\N,\; G=\bigcup_k Q_{s,k} \\
&r\leq s \implies Q_{r,l}\subset Q_{s,k} \text{ or } Q_{r,l}\cap Q_{s,k}= \emptyset\\
&\forall(r,l), \forall s>r\; \exists! k\in\Z \text{ such that } Q_{r,l}\subset Q_{s,k}\\
\label{diameter}
&\text{Diameter } Q_{s,k}\leq C_1A^s\\
&\text{Each }Q_{s,k}\text{ contains some ball of radius }a_0 A^s.
\end{eqnarray}
Because $G$ has a polynomial rate of growth, $\rho$ is a doubling metric, and thus we can prove the Vitali Covering Lemma and the Hardy-Littlewood Maximal Inequality on $G$.  Using a standard stopping-time argument, we can then define a suitable discrete Calderon-Zygmund decomposition on $G$ with the dyadic cubes.
\\
\\ Fix $\lambda>0$.  We take $\varphi ={\mathfrak g}+b$, where $\|{\mathfrak g}\|_\infty\leq\lambda$ and $b=\displaystyle\sum_{(s,k)\in\B} b_{s,k}$ for some index set $\B\subset \N^2$, where $b_{s,k}$ is supported on $Q_{s,k}$, $\{Q_{s,k}:(s,k)\in\B\}$ is a disjoint collection, $\|b_{s,k}\|_1\leq\lambda|Q_{s,k}|$ and $\displaystyle \sum_{(s,k)\in\B}|Q_{s,k}|\leq \frac{C}\lambda\| \varphi\|_1$ ($C$ independent of $\varphi$ and $\lambda$).  Let $b_s=\displaystyle\sum_k b_{s,k}$.
\\
\\We further decompose $b_{s,k}=b^{(j)}_{s,k}+B^{(j)}_{s,k}$, where $b^{(j)}_{s,k}=b_{s,k}\1(|b_{s,k}|>\lambda r_j)$.  Define $b^{(j)}_s, B^{(j)}_s, b^{(j)}, B^{(j)}$ by summing over one or both indices, respectively.
\\
\\ We will divide $B^{(j)}=\sum_s B^{(j)}_s$ into two parts, splitting at the index $s(j):=\min\{s: A^s\geq R_j\}$.
\\
\\Now $\{g: \sup_j | \varphi\ast\mu_j (g)|> 5\lambda\}\subset$
\begin{eqnarray*}
\{g: \sup_j |{\mathfrak g} \ast\mu_{j} (g)|> \lambda\}\cup\{g: \sup_j |b^{(j)}\ast\mu_j  (g)|> \lambda\}\cup\{g: \sup_j |B^{(j)}\ast(\mu_{j}-\nu_j) (g)|> \lambda\}\\
 \cup\{g: \sup_j |\left(\sum_{s=s(j)}^\infty B_{s}^{(j)} \right)\ast\nu_j(g)|> \lambda\}\cup\{g: \sup_j | \left(\sum_{s=0}^{s(j)-1}B_{s}^{(j)} \right)\ast\nu_j(g)|> \lambda\}
\end{eqnarray*}
$$=  E_1\cup E_2\cup E_3\cup E_4\cup E_5.$$
\\ \noindent By the weak $(p,p)$ inequality (if $p<\infty$), $|E_1|\leq C\lambda^{-p}\|{\mathfrak g}\|_p^p\leq C\lambda^{-p}\|{\mathfrak g}\|_\infty^{p-1}\|{\mathfrak g}\|_1\leq C\lambda^{-1}\|\varphi\|_1$; if $p=\infty$, consider instead $\{g: \sup_j |{\mathfrak g} \ast\mu_{j} (g)|> C_\infty\lambda\}=\emptyset$ since $\|\sup_j |{\mathfrak g} \ast\mu_{j}|\|_\infty\leq C_\infty\|{\mathfrak g}\|_\infty=C_\infty\lambda$.
\\
\\Next,
\begin{eqnarray*}
|E_2|\leq\sum_j |\{g:|b^{(j)}\ast\mu_{j} (g)|>0\}|&\leq&\sum_j |\text{supp }\mu_{j}|\cdot|\{g: |b(g)|>\lambda r_j\}|\\
&=&\sum_j r_j \sum_{k\geq j} |\{g:\lambda r_k<|b(g)|\leq \lambda r_{k+1}\}|\\
&=&\sum_{k} |\{g:\lambda r_k<|b(g)|\leq \lambda r_{k+1}\}| \sum_{j\leq k}r_j\\
&\leq&\frac{C_0}\lambda \sum_k \lambda r_k |\{g:\lambda r_k<|b(g)|\leq \lambda r_{k+1}\}|;
\end{eqnarray*}
 now note that this sum is a lower sum for $|b|$, and we have $|E_2|\leq C_0\lambda^{-1}\|b\|_1\leq\frac{C}\lambda\|\varphi\|_1$.
\\
\\For $E_3$, $|B^{(j)}\ast(\mu_j-\nu_j) (g)|\leq |B^{(j)} |\ast|\mu_j-\nu_j|(g)\leq|b|\ast|\mu_j-\nu_j|(g)$, so by the weak $(1,1)$ inequality, $|E_3|\leq |\{\sup_j |b|\ast|\mu_j-\nu_j|(g)>\lambda\}|\leq \frac{C}\lambda \|b\|_1\leq \frac{C}\lambda\|\varphi\|_1.$
\\
\\ To bound $|E_4|$, for all $s\geq s(j)$, $\text{supp }(B_{s,k}^{(j)}\ast\nu_j )\subset Q_{s,k}+\text{supp }\nu_j\subset Q^*_{s,k}:=\{g: \rho(g, Q_{s,k})\leq A^s\}$, so
\begin{eqnarray*}
E_4\subset\displaystyle \bigcup_{s,k}\bigcup_{j:R_j\leq A^{s+1}}\text{supp }(B_{s,k}^{(j)}\ast\nu_j )\leq\sum_{(s,k)\in\B} C|Q_{s,k}|\leq\frac{C}\lambda\|\varphi\|_1.
\end{eqnarray*}
We have thus reduced the problem to obtaining a bound on $|E_5|$.
\begin{lemma}
Let $B_s^{(j)}$ be as above, and assume the $\nu_j$ satisfy (\ref{convo}).  For $0\leq s<s(j)$,
\begin{eqnarray*}
\|B_{s}^{(j)}\ast\nu_j \|_{\ell^2(G)}^2\leq C r_j^{-1}\|B_{s}^{(j)}\|_2^2 + C\lambda 2^{-\epsilon j}\|B_{s}^{(j)}\|_1
\end{eqnarray*}
and for $0\leq s_1<s_2<s(j)$,
\begin{eqnarray*}
|\langle B_{s_1}^{(j)}\ast\nu_j ,B_{s_2}^{(j)}\ast\nu_j \rangle_{\ell^2(G)}|\leq C \lambda 2^{-\epsilon j} \|B_{s_2}^{(j)}\|_1.
\end{eqnarray*}
\end{lemma}
\begin{proof}
We first restrict the supports of the $B_s$; we assume there is a $Q_{s(j),k_0}$ such that $Q_{s,k}\subset Q_{s(j),k_0}$ for all $(s,k)\in\B$ with $s<s(j)$.  Then $\|B_s^{(j)}\|_1\leq \|b_s\|_1\leq\sum_{(s,k)\in\B}\lambda|Q_{s,k}|\leq\lambda |Q_{s(j),k_0}|\leq C\lambda R_j^{d}$, and thus
\begin{eqnarray*}
|\langle B_{s_1}^{(j)}\ast\nu_j ,B_{s_2}^{(j)}\ast\nu_j \rangle|&=&|\langle B_{s_1}^{(j)}\ast\nu_j\ast \tilde\nu_j, B_{s_2}^{(j)}\rangle|\\
&\leq& Cr_j^{-1}|\langle B_{s_1}^{(j)}, B_{s_2}^{(j)}\rangle|+ CR_j^{-d}2^{-\epsilon j}\|B_{s_1}^{(j)}\|_1 \|B_{s_2}^{(j)}\|_1\\
&\leq& Cr_j^{-1}|\langle B_{s_1}^{(j)}, B_{s_2}^{(j)}\rangle|+ C\lambda 2^{-\epsilon j}\|B_{s_2}^{(j)}\|_1.
\end{eqnarray*}
Now this first term is 0 if $s_1\neq s_2$, and $Cr_j^{-1}\|B_{s_1}^{(j)}\|_2^2$ if $s_1=s_2$.
\\
\\We remove the assumption on the supports by noting that if the distance between the supports of $\varphi_1$ and $\varphi_2$ is greater than $2R_j$, then $\langle \varphi_1\ast\nu_j, \varphi_2\ast\tilde\nu_j\rangle=0$.  Thus if we decompose each $B_s=\sum_k B_s\1(Q_{s(j),k})$ and decompose the inner products accordingly, all but finitely many of the terms (a number independent of $j$) will vanish; and those remaining can be estimated in this way.
\end{proof}
\noindent Now by Chebyshev's Inequality,
\begin{eqnarray}
\label{money}
\lambda^2|\{g: \sup_j |\sum_{s=0}^{s(j)-1} B_{s}^{(j)}\ast\nu_j (g)|>\lambda\}|\leq\sum_g\sup_j |\sum_{s=0}^{s(j)-1} B_{s}^{(j)}\ast\nu_j (g)|^2\leq\sum_j \|\sum_{s=0}^{s(j)-1} B_{s}^{(j)}\ast\nu_j\|_2^2 
\end{eqnarray}
\begin{eqnarray*}
&&\leq \sum_j\sum_{\scriptsize \begin{array}{c} s_1,s_2:\\ 0\leq s_1,s_2 < s(j) \end{array}} |\langle B_{s_1}^{(j)}\ast\nu_j ,B_{s_2}^{(j)}\ast\nu_j \rangle_{\ell^2(G)}|\\
&&\leq \sum_j\sum_{s=0}^{s(j)-1} \left(C r_j^{-1}\|B_{s}^{(j)}\|_2^2 + C\lambda 2^{-\epsilon j}\|B_{s}^{(j)}\|_1\right)+2\sum_j\sum_{\scriptsize \begin{array}{c} s_1,s_2:\\  0\leq s_1<s_2 < s(j) \end{array}} C \lambda 2^{-\epsilon j} \|B_{s_2}^{(j)}\|_1\\
&&\leq \sum_{s=0}^\infty \sum_{j=1}^\infty C\lambda (1+j)2^{-\epsilon j}\|B_s^{(j)}\|_1+ \sum_j\sum_{s=0}^{s(j)-1} Cr_j^{-1}\|B_{s}^{(j)}\|_2^2\\
&&\leq \sum_{s=0}^\infty C\lambda \|b_s\|_1+\sum_j\sum_{s=0}^{s(j)-1} Cr_j^{-1}\|B_{s}^{(j)}\|_2^2.
\end{eqnarray*}
The first term is $\leq C\lambda\|\varphi\|_1$ as desired.  For the second term, note that 
$$\sum_{j\leq k} r_j\leq C_0r_k\;\forall k\in\N\implies \exists N \text{ s.t. } r_{j+n}\geq 2r_j \forall j\in\N,n\geq N\implies \sum_{j=k}^\infty r_j^{-1}\leq Cr_k^{-1}.$$
Since the $Q_{s,k}$ are disjoint, for a fixed $g\in Q_{s_0,k_0}$,
\begin{eqnarray*}
\sum_j\sum_{s=0}^{s(j)-1}  r_j^{-1}|B_{s}^{(j)}(g)|^2\leq\sum_{\scriptsize \begin{array}{c} j:\\ \lambda r_j\geq |b_{s_0}(g)| \end{array}}r_j^{-1}|b_{s_0}(g)|^2\leq C\lambda|b_{s_0}(g)|=C\lambda|b(g)|
\end{eqnarray*}
so $\sum_j\sum_{s=0}^{j-1} Cr^{-j}\|B_{s}^{(j)}\|_2^2\leq C\lambda\|b\|_1\leq C\lambda\|\varphi\|_1$ and the proof of (\ref{sup}) is complete.
\end{proof}
\noindent Having established Proposition \ref{shine on}, it remains to show that the random measures $\mu_j^{(\omega)}$ and $\nu_j^{(\omega)}$ satisfy the assumptions with probability 1.  Note first that  $r_j=|\text{supp }\mu_j^{(\omega)}|=\sum_{g\in\S_{2^j}}\xi_g(\omega)={\mathbf\Theta}(\beta(2^j))={\mathbf\Theta}(2^{(d-\alpha)j})$ on $\Omega_1$, and $\nu_j^{(\omega)}$ is supported on $\S_{2^j}$ with $\rho$-diameter at most $R_j=2^{j+1}$. We must prove the bound (\ref{convo}) on $\nu_j^{(\omega)}\ast \tilde\nu_j^{(\omega)}$.
\begin{lemma}
\label{cancels}
Let $G$ be a group and $E$ a finite subset.  Let $\{  X_g\}_{g\in E}$ be independent random variables with $|X_g|\leq1$ and $\E X_g=0$.  Assume that $\sum_{g\in E} (\Var X_g)^2\geq1$.  Let $X$ be the random $\ell^1(G)$ function $\sum_{g\in E} X_g\delta_g$.  Let $G^\times$ denote $G\setminus\{e\}$.  Then for any $\theta>0$,
\begin{eqnarray}
\P\left(\| X\ast\tilde X\|_{\ell^\infty(G^\times)}\geq \theta(\sum_{g\in E} (\Var X_g)^2)^{1/2}\right)\leq 6|{E}|^2\max(e^{-\theta^2/36},e^{-\theta/6}).
\end{eqnarray}
\end{lemma}
\begin{proof}
For $h\neq e$,
$$X\ast\tilde X(h)=\sum_{g\in E\cap h^{-1} E}X_gX_{gh}=\sum_{g\in E\cap h^{-1} E} Y_g$$
where $\E Y_g=0$ and $|Y_g|\leq1$.  We want to apply Chernoff's Inequality, but the $Y_g$ are not independent.
\\
\\We can, however, partition $E\cap h^{-1}E$ into at most three subsets $E_1,E_2,E_3$, in each of which the $Y_g$ are independent.  To see this, note that we can make a directed graph with vertex set $E$ and edge set $\{(g,hg):g, hg\in E\}$; and that the components of this graph are paths or cycles.  Thus we can three-color this graph; and within each resulting $E_i$, the $Y_g$ depend on distinct independent random variables, so they are independent.
\\
\\ Now $\displaystyle \sum_{g\in E_i} Y_g$ has variance $\displaystyle \sigma^2=\sum_{g\in E_i}\Var X_g \Var X_{gh}\leq\sum_{g\in E_i}(\Var X_g)^2\leq \sum_{g\in E}(\Var X_g)^2$
by H\"older's Inequality.  Chernoff's Inequality (Theorem 1.8 in \cite{TV}) gives us $\displaystyle \P(|\sum_{g\in E_i} Y_g|\geq\lambda\sigma)\leq2\max(e^{-\lambda^2/4},e^{-\lambda\sigma/2})$.
\\
\\ Take $\lambda=\theta\sigma^{-1}(\sum_{g\in E} (\Var X_g)^2)^{1/2}$; then $\lambda\geq\theta$ and $\lambda\sigma=\theta(\sum_{g\in E} (\Var X_g)^2)^{1/2}\geq\theta$, so
$$\P(|X\ast\tilde X(h)|\geq 3\theta(\sum_{g\in E} (\Var X_g)^2)^{1/2})\leq\sum_{i=1}^3\P(|\sum_{E_i} Y_g|\geq\lambda\sigma)\leq6\max(e^{-\theta^2/4},e^{-\theta/2}).$$
Since this holds for each $h\neq e$ and $|\text{supp }X\ast\tilde X|\leq |E|^2$, the conclusion follows (after replacing $3\theta$ with $\theta$).
\end{proof}
\begin{corollary}
\label{bound}
Let $\nu_j^{(\omega)}$ be the random measure defined as before, $0<\alpha<d/2$ and $\kappa>0$.  Then there is a set $\Omega_3\subset\Omega_2$ with $\P(\Omega_3=1)$ such that for each $\omega\in\Omega_3$,
\begin{eqnarray}
\nu_j^{(\omega)}\ast \tilde\nu_j^{(\omega)}= O_\omega(\beta(2^j)^{-1})\delta_e+O_\omega(\beta(2^j)^{-2}(\sum_{g\in\S_{2^j}}\tau_g^2)^{1/2}2^{\kappa j}).
\end{eqnarray}
\end{corollary}
\begin{proof}
For the bound at the identity $e$, we use the fact that 
\begin{eqnarray*} \nu_j^{(\omega)}\ast\tilde\nu_j^{(\omega)}(e)=\beta(2^j)^{-2}\sum_{g\in\S_{2^j}}\eta_g^2(\omega)&=&\beta(2^j)^{-2}\sum_{g\in\S_{2^j}}\left(\tau_g^2(1-\xi_g(\omega))+(1-\tau_g)^2\xi_g(\omega)\right)\\
&\leq&\beta(2^j)^{-2}\sum_{g\in\S_{2^j}}(\tau_g+\xi_g(\omega))=2\beta(2^j)^{-1}+\beta(2^j)^{-2}\sum_{g\in\S_{2^j}}\eta_g(\omega)
\end{eqnarray*}
so that
$$\P(\nu_j^{(\omega)}\ast\tilde\nu_j^{(\omega)}(e)>3\beta(2^j)^{-1})\leq\P(\sum_{g\in\S_{2^j}} \eta_{g}(\omega) >\beta(2^j))\leq 2\exp(-\frac12\beta(2^j))$$ for $j$ sufficiently large, by Chernoff's inequality.  The Borel-Cantelli Lemma implies that $\nu_j^{(\omega)}\ast\tilde\nu_j^{(\omega)}(e)\leq 3\beta(2^j)$ for $j$ sufficiently large (depending on $\omega$), so there exists $C_\omega$ with $0\leq\nu_j^{(\omega)}\ast\tilde\nu_j^{(\omega)}(e)\leq C_\omega \beta(2^j)$ for all $j$.
\\
\\ For the other term, we note that $\Var \eta_g\leq \tau_g$, so we set $\theta=2^{\kappa j}$ and apply Lemma \ref{cancels}:
\begin{eqnarray*}
\P\left(\beta(2^j)^2\|\nu_j^{(\omega)}\ast\tilde\nu_j^{(\omega)}\|_{\ell^\infty(G^\times)}\geq 2^{\kappa j}(\sum_{g\in\S_{2^j}}\tau_g^2)^{1/2}\right)\leq C2^{2dj}\exp(-2^{\kappa j}/2)
\end{eqnarray*}
which sum over $j$.  The Borel-Cantelli Lemma again proves the bound holds with probability 1.
\end{proof}
\noindent Note that $\displaystyle\sum_{g\in\S_{2^j}}\tau_g^2={\mathbf\Theta}(2^{(d-2\alpha)j});$ thus for $\alpha<d/2$, $\beta(2^j)^{-2}(\sum_{g\in\S_{2^j}}\tau_g^2)^{1/2}2^{\kappa j}=O(2^{(-\frac{3d}2+\alpha+\kappa)j})=O(R_j^{-d}2^{-\epsilon j})$ for $\kappa,\epsilon$ small.  Therefore the measures $\nu_j^{(\omega)}$ satisfy the bound (\ref{convo}) , for all $\omega\in\Omega_3$.  Since $\mu_j^{(\omega)}-\nu_j^{(\omega)}=\E\mu_j$ is a weighted average of the nonnegative averages in (\ref{costello}), Theorem K2 implies $\|\sup_j |\varphi\ast\E\mu_j|\|_{1,\infty}\leq C\|\varphi\|_1$; and the $\ell^\infty$ maximal inequality for $\mu_j^{(\omega)}$ is trivial.  Thus Proposition \ref{shine on} applies, and we have proved Theorem \ref{L1}.
\begin{remark}
Our estimate on $\nu_j^{(\omega)}\ast\tilde\nu_j^{(\omega)}$ is nearly optimal; for $h\neq e$ in the basis set $\A$, $$\E|\nu_j^{(\omega)}\ast\tilde\nu_j^{(\omega)}(h)|^2\geq \beta(2^j)^{-4}\sum_{g\in S_{2^j}\cap h^{-1}S_{2^j}}\tau_g(1-\tau_g)\tau_{gh}(1-\tau_{gh})\geq c 2^{(-3d+2\alpha)j},$$
which strongly suggests that this method cannot work for random sets if $\alpha\geq d/2$.
\end{remark}

\section{More General Averages}
In the preceding sections, we have taken very particular averages, but the machinery of the proof allows substantially more flexibility. If we have a sequence of sets $\S_N\subset G$, a sequence of probabilities $\{\tau_g:g\in G\}$ and independent Bernoulli random variables $\xi_g$ with means $\tau_g$, a group action $(X,\F,m,\{T_g\})$ and $f\in L^1(X)$, we may define as before
\begin{eqnarray*}
\beta(N)&:=&\sum_{g\in\S_N}\tau_g\\
\sigma_N f(x) &:=& {\beta(N)}^{-1}\sum_{g\in\S_N}\tau_gf(T_gx)\\
A_N^{(\omega)}f(x)&:=&{\beta(N)}^{-1} \sum_{g\in\S_N}\xi_g(\omega)f(T_gx)\\
\mu_N^{(\omega)}(g)&:=& \beta(N)^{-1}\xi_g(\omega)\1_{\S_N}(g)\\
\E\mu_N(g)&:=& \beta(N)^{-1}\tau_g\1_{\S_N}(g).
\end{eqnarray*}
When working with block sequences, the sumsets $(\S_N^{-1}\S_N)^M:=\{g_1^{-1}h_1\dots g_N^{-1}h_N: g_i,h_i\in \S_N \forall i\}$ can become too large if they are not first partitioned into their individual blocks.  For this reason, we introduce one final bit of notation.  Let $\{\Eset_i:i\in\N\}$ be a collection of nonempty subsets of $G$, and define 
\begin{eqnarray*}
\beta'(i)&:=&\sum_{g\in\Eset_i}\tau_g\\
\nu_i^{(\omega)}&:=&\beta'(i)^{-1}\sum_{g\in\Eset_i}(\xi_g-\tau_g).
\end{eqnarray*}
\begin{remark}
Note that in general, if  $\S_N\subset\S_{N+1}$ and $\beta(N)\to\infty$, then ${\beta(N)}^{-1}\sum_{g\in\S_N}\xi_g \to1$ with probability 1.  Let $N_j= \min\{n:\beta'(n)\geq j\}$; then Chernoff's Inequality establishes that for any $K\in\N$,
$$\P(|\sum_{g\in\S_{N_j}}\xi_g -\tau_g|\geq \frac{\beta(N_j)}{K})\leq 2\exp(-j/4K^2).$$
After applying Borel-Cantelli and intersecting these sets $\Omega_K$, we conclude that ${\beta'(N_j)}^{-1}\sum_{g\in\S'_{N_j}}\xi_g \to1$ on a set of probability 1; we similarly conclude the same along the sequence $\{N_j-1\}$.  Then for $N_j\leq N<N_{j+1},$ 
$$\frac{j}{j+1}\beta(N_j)^{-1}\sum_{g\in\S_{N_j}}\xi_g \leq\beta(N)^{-1}\sum_{g\in\S_N}\xi_g\leq\frac{j+1}{j}\beta(N_{j+1}-1)^{-1}\sum_{g\in\S_{(N_{j+1}-1)}}\xi_g$$
so we have convergence for every $N$.
\\
\\ If the $S_N$ are not nested, but $\beta(N)\geq cN^\epsilon$, then a similar calculation proves ${\beta(N)}^{-1}\sum_{g\in\S_N}\xi_g \to1$ with probability 1.
\end{remark}
We can now strengthen our results from Sections 3 and 4 in the following fashion:
\begin{theorem}
\label{L2'}
Let $G$ be an infinite discrete group with polynomial growth, $\{\tau_g: g\in G\}$ a sequence of probabilities, and $\S_N\subset G$ a sequence of sets with $\beta(N)\to\infty$.
\\
\\Let $\{\Eset_i:i\in\N\}$ be a collection of nonempty subsets of $G$ such that for each $N>0$, $\exists \,{\I}_N\subset \N$ such that $\{\Eset_i:i\in\I_N\}$ is a partition of $\S_N$.  If for some $M\in\N$ and $\epsilon>0$, 
\begin{eqnarray*}
\beta'(i)^{-2M}|(\Eset_i^{-1}\Eset_i)^M|=O(i^{-2M-1-\epsilon}),
\end{eqnarray*}
then there exists a set $\Omega_2\subset\Omega$ with $\P(\Omega_2)=1$ such that for every $\omega\in\Omega_2$ and every measure-preserving group action $(X,\F,m,\{T_g\})$,  $A_N^{(\omega)}f-\sigma_Nf\to0$ in $L^2(X)$ and a.e. for every $f\in L^2(X)$.
\end{theorem}
\begin{proof}
We again transfer the problem to $\ell^2(G)$ by Lemma \ref{transfer}.  Since $\mu_N^{(\omega)}-\E\mu_N=\sum_{i\in\N}a_{i,N}\nu_i$ where $a_{i,N}\geq0$, $\sum_{i\in\N}a_{i,N}=1$ for all $N$ and $\lim_{N\to\infty}a_{i,N}=0$ for all $i$ (since $\beta(N)\to\infty$), it suffices to prove that with probability 1 there is a sequence $C_{k,\omega}\to0$ such that
\begin{eqnarray}
\|\sup_{i\geq k} |\psi\ast\nu_i^{(\omega)}|\|_2 \leq C_{k,\omega}\|\psi\|_2\; \forall \psi\in\ell^2(G).
\end{eqnarray}
The argument for $\nu_j$ in Section 3 applies in the same way to these measures, and we find
\begin{eqnarray*}
\P(\|(\tilde\nu_i^{(\omega)}\ast\nu_i^{(\omega)})^M\|_1>\lambda)\leq\lambda^{-2}|(\Eset_i^{-1}\Eset_i)^M|\cdot C_M\beta'(i)^{-2M}
\end{eqnarray*}
which leads to the desired bound.
\end{proof}
\begin{corollary}
\label{L2sub}
Let $\{n_k\}$ be an increasing universally $L^2$-good sequence in $\N$ which is polynomially bounded; let $0<\alpha<1$, and let $\xi_k$ be independent Bernoulli random variables with $\P(\xi_k=1)=k^{-\alpha}$.  Then with probability 1, $\{n_k:\xi_k(\omega)=1\}$ is also universally $L^2$-good.
\end{corollary}
\begin{proof}
We set $\Eset_N=\S_N=[0,a_N]\cap\{n_k\}$ for an increasing sequence $\{a_j\}\subset \N$ which is superpolynomial and has $\frac{a_{j+1}}{a_j}\to1$, then apply Theorem \ref{L2'}; we then pass from this subsequence of the averages in the usual way.
\end{proof}
\noindent The construction in Theorem \ref{L1} requires stronger assumptions in order to bound the term $E_4$Ñ namely, that $\bigcup_{j=1}^N\S_j$ can be covered by a fixed number of sets with appropriate diameters $R_N$.  This leads to the following formulation:
\begin{theorem}
\label{L1'}
Let $G$ be a group with polynomial growth of degree $d$, $\{\tau_g: g\in G\}$ a sequence of probabilities, and $\S_N\subset G$ a sequence of sets with $\sum_{N=1}^M \beta(N)\leq C\beta(M)\;\forall N$, such that we have the $\ell^1(G)$ weak maximal inequality $\|\sup_N |\psi\ast\E\mu_N|\|_{1,\infty}\leq C\|\psi\|_1\;\forall \psi\in \ell^1(G)$.
\\
\\ Say there exist $K\in \N$, $\epsilon>0$, a sequence $\{R_N\}\subset \N$ and a sequence of sets $\{\Eset_{i,N}\subset G:1\leq i\leq K;N\in\N\}$ such that
\begin{eqnarray*}
&\bigcup_{j=1}^N\S_j\subset \bigcup_{i=1}^K \Eset_{i,N}\\
&\text{diameter }\Eset_{i,N}\leq R_N\\
&\beta(N)^{-2}(\displaystyle\sum_{g\in \S_N}\tau_g^2)^{1/2}R_N^d=O_{j\to\infty}(2^{-\epsilon N}).
\end{eqnarray*}
Then there exists a set $\Omega_3\subset\Omega$ with $\P(\Omega_3)=1$ such that for every $\omega\in\Omega_3$ and every measure-preserving group action $(X,\F,m,\{T_g\})$, $A_N^{(\omega)}f-\sigma_Nf\to0$ in $L^1(X)$ and a.e. for every $f\in L^1(X)$.
\end{theorem}
\begin{proof}
We first note that the $L^2$ result holds with these assumptions, so it suffices to prove the maximal inequality $\|\sup_N |\varphi\ast\mu_{N}^{(\omega)}|\|_{1,\infty}\leq C_\omega\|\varphi\|_1$.  For the set $E_4$ in Proposition \ref{shine on}, we note that for $N(s):=\min\,\{N:R_N\geq A^s\}$,
$$Q_{s,k}+\bigcup_{j<N(s)}\text{ supp }\nu_j^{(\omega)}\subset Q_{s,k}^*:=Q_{s,k}+ \bigcup_{i=1}^K \Eset_{i,N(s)},$$
and $|Q_{s,k}^*|\leq CA^s$.  The rest of the proof is unchanged.
\end{proof}
\begin{corollary}
Let $\{n_k\}$ be an increasing universally $L^1$-good sequence in $\Z$ with $n_k=O(k^{\frac32-\delta})$ for some $\delta>0$; let $0<\alpha<\delta$, and let $\xi_k$ be independent Bernoulli random variables with $\P(\xi_k=1)=k^{-\alpha}$.  Then with probability 1, $\{n_k:\xi_k(\omega)=1\}$ is also universally $L^1$-good.
\end{corollary}
\begin{proof}
By Corollary \ref{L2sub}, we only need to prove the weak (1,1) maximal inequality, and we have the requisite inequality for $\E\mu_N$ because $\{n_k\}$ is universally $L^1$-good.  Let $\Eset_{1,N}=\{n_k:1\leq k\leq 2^N\}$; then $R_N=O(2^{(\frac32-\delta)N})$, $\beta(N)={\mathbf \Theta}(2^{(1-\alpha)N})$, and $\sum_{k=1}^N\tau_k^2={\mathbf\Theta}(2^{(1-2\alpha)N})$.  Thus Theorem \ref{L1'} applies.
\end{proof}
\noindent We can now prove the final theorem from Section 1:
\begin{theorem}
\label{faster}
For every $F:\N\to\R^+$, there exists a universally $L^1$-good sequence $\{n_k\}$ with $n_k\geq F(k)$ and Banach density 0.
\end{theorem}
\begin{proof}
It was proved in \cite{BL} that if $\{n_k\}$ is a block sequence $\bigcup_j [v_j, w_j)\cap\Z$ with $v_j<w_j<v_{j+1}$ such that $w_j-v_j\geq v_{j-1}$ for all $j$, then $\{n_k\}$ is universally $L^1$-good.  Thus we can take such a sequence which already grows faster than $F(k)$, and take a random subsequence $\{n_k: \xi_k(\omega)=1\}$ with $\tau_k=k^{-\alpha}$, $\alpha<1/2$.
\\
\\ Note that for any $N$, if $n_N\in[v_j,w_j)$, we let $\Eset_{1,N}=[0,v_{j-2}),\, \Eset_{2,N}=[v_{j-2},w_{j-2}),\, \Eset_{3,N}=[v_{j-1},w_{j-1})$ and $\Eset_{4,N}=[v_j,n_N]$.  Then each $|\Eset_{i,N}|\leq N$, and thus $|\Eset_{i,N}^{-1}\Eset_{i,N}|\leq 3N$, while $\beta(N)={\mathbf \Theta}(N^{1-\alpha})$.  Thus if we take a superpolynomial sequence $N_j$ with $\frac{N_{j+1}}{N_j}\to1$, Theorem \ref{L2'} applies; passing from this subsequence in the usual way, we find $\{n_k: \xi_k(\omega)=1\}$ is universally $L^2$-good.
\\
\\ We then take $N=2^n$ and apply Theorem \ref{L1'} with these $\Eset_{i,N}$ ($R_N=N$) to obtain the $L^1$ maximal inequality along this subsequence, which suffices to prove that with probability 1, $\{n_k: \xi_k(\omega)=1\}$ is universally $L^1$-good as desired.
\end{proof}
\noindent It remains, finally, to note that such randomly generated sequences and subsequences indeed have Banach density 0 (with probability 1) if the $\tau_n$ decrease according to a power law.  Conveniently enough, a converse result also holds:
\begin{proposition}
\label{banach}
Let $\{\tau_n\}$ be a decreasing sequence of probabilities, and let $\xi_n$ be independent Bernoulli random variables with $\P(\xi_k=1)=k^{-\alpha}$.  Then if $\tau_n= O(n^{-\alpha})$ for some $\alpha>0$, the sequence of integers $\{n:\xi_n=1\}$ has Banach density 0 with probability 1 in $\Omega$; otherwise, it has Banach density 1 with probability 1 in $\Omega$.
\end{proposition}
\begin{proof}
It is elementary to show that 
\begin{eqnarray}
\label{easy}
2^{-r}\tau_{r(n+1)}^m\leq\P\left( \sum_{j=rn}^{r(n+1)-1}\xi_j \geq m\right)\leq2^r\tau_{rn}^m.
\end{eqnarray}
(We majorize or minorize the $\xi_j$ by i.i.d. Bernoulli variables and use the Binomial Theorem.) Then if $\tau_n= O(n^{-\alpha})$, let $K>0$ and fix $m,r\in\N$ such that $m\alpha>1$ and $r>mK$; the probabilities above are then summable, so the first Borel-Cantelli Lemma implies that on a set $\Omega_K$ of probability 1 in $\Omega$, there exists an $M_\omega$ such that for all $n\geq M_\omega$, $\sum_{j=rn}^{r(n+1)-1}\xi_j < m<\frac{r}{K}$; then it is clear that $\{n:\xi_n=1\}$ has Banach density less than $3K^{-1}$.  Let $\Omega'=\bigcap_K \Omega_K$; then $\P(\Omega')=1$ and $\{n:\xi_n=1\}$ has Banach density 0 on $\Omega'$.
\\
\\ For the other implication, note that if $\tau_n\neq O(n^{-1/R})$, there exists a sequence $n_k$ with $n_{k+1}\geq2n_k$ such that $\tau_{n_k}\geq n_k^{-1/R}$; then $$\sum_{n=1}^\infty \tau_{Rn}^R\geq R^{-1}\sum_{n=2}^\infty \tau_n^R\geq R^{-1}\sum_{k=2}^\infty(n_k-n_{k-1})\tau_{n_k}^R\geq R^{-1}\sum_{k=2}^\infty \frac12=\infty.$$
Thus the probabilities in (\ref{easy}) are not summable in $n$, for $m=r=R$.  Since the variables $\xi_n$ are independent, the second Borel-Cantelli Lemma implies that there is a set $\tilde\Omega_R$ of probability 1 on which $\{n:\xi_n(\omega)=1\}$ contains infinitely many blocks of $R$ consecutive integers.  Therefore if $\tau(n)\neq O(n^{-\alpha})$ for every $\alpha>0$, let $\tilde\Omega'=\bigcap_R \tilde\Omega_R$; on this set of probability 1, $\{n:\xi_n=1\}$ has Banach density 1.
\end{proof}
\noindent The author thanks his dissertation advisor, M. Christ, for consultation and assistance throughout the composition of this paper, and M. Wierdl and J. Rosenblatt for many comments and suggestions.

\end{document}